\documentclass[english,12pt,a4paper]{amsart}

\title{Galois groups over rational function fields over skew fields}

\author{Gil Alon}
\email{gilal@openu.ac.il}

\author{Fran\c cois Legrand}
\email{francois.legrand@tu-dresden.de}

\author{Elad Paran}
\email{paran@openu.ac.il}

\address{Department of Mathematics and Computer Science, the Open University of Israel, Ra'anana 4353701, Israel}

\address{Institut f\"ur Algebra, Fachrichtung Mathematik, TU Dresden, 01062 Dresden, Germany}

\address{Department of Mathematics and Computer Science, the Open University of Israel, Ra'anana 4353701, Israel}

\usepackage{fullpage}
\usepackage{amsmath}
\usepackage{amssymb}
\usepackage{hyperref}

\newenvironment{poliabstract}[1]
{\renewcommand{\abstractname}{#1}\begin{abstract}}
{\end{abstract}}

\newcommand{\id}{{\rm id}}
\newcommand{\Hh}{\mathbb{H}}
\newcommand{\Q}{\mathbb{Q}}
\newcommand{\Qq}{\mathbb{Q}}
\newcommand{\C}{\mathbb{C}}
\newcommand{\R}{\mathbb{R}}
\newcommand{\Rr}{\mathbb{R}}
\newcommand{\N}{\mathbb{N}}

{\theoremstyle{plain}
\newtheorem{theorem}{Theorem}[section]

\newtheorem{corollary}[theorem]{Corollary}
\newtheorem{proposition}[theorem]{Proposition}}

{\theoremstyle{remark}
\newtheorem{remark}[theorem]{Remark}
}

\begin{document}

\maketitle

\vspace{-4mm}

\begin{abstract}
Let $H$ be a skew field of finite dimension over its center $k$. We solve the Inverse Galois Problem over the field of fractions $H(X)$ of the ring of polynomial functions over $H$ in the variable $X$, if $k$ contains an ample field.
\end{abstract}

\begin{poliabstract}{R\'esum\'e}
Soit $H$ un corps gauche de dimension finie sur son centre $k$. Nous r\'esolvons le Probl\`eme Inverse de Galois sur le corps des fractions $H(X)$ de l'anneau des fonctions polynomiales en la variable $X$ et \`a coefficients dans $H$, si $k$ contient un corps ample.
\end{poliabstract}

\section{Introduction}

The Inverse Galois Problem over a field $k$ asks whether every finite group occurs as the Galois group of a Galois field extension of $k$. Hilbert showed in 1892, via his cele\-bra\-ted irreducibility theorem, that this problem over the field $\Qq$ of rational numbers is equi\-valent to the same problem over the field $\Q(t)$ of rational functions over $\Q$. While the problem is wide open over $\Q(t)$, it is known to have an affirmative answer over many other function fields, e.g., over the field $\C(t)$ of complex rational functions, as a consequence of Riemann's Existence Theorem. 

The aim of this note is to contribute to inverse Galois theory over skew fields, following a first work on this topic by Deschamps and Legrand (see \cite{DL19}). In this more general context, given skew fields (equivalently, division rings) $H \subseteq M$, the extension $M/H$ is said to be {\it{Galois}} if every element of $M$ which is fixed under any automorphism of $M$ fixing $H$ pointwise lies in $H$. See \cite[\S3.3]{Coh95} for more on Galois theory over skew fields.

Let $H$ be a skew field, and let $H[X]$ denote the ring of all polynomial functions over $H$ in the variable $X$. That is, $H[X]$ is the ring of all functions from $H$ to $H$ that can be expressed by sums and products of the variable $X$ and elements of $H$. We observe that, if $H$ is of finite dimension over its center $k$ and if $k$ is infinite, then $H[X]$ has a classical (right) field of fractions, denoted by $H(X)$. See \S\ref{sec:basics} for more details.

In the sequel, we solve the Inverse Galois Problem over the skew field $H(X)$, if the center of $H$ contains an ample field:

\begin{theorem}\label{thm:main} 
Let $H$ be a skew field of finite dimension over its center $k$. If $k$ contains an ample field, then every finite group is the Galois group of a Galois extension of $H(X)$. 
\end{theorem}

\noindent
Recall that a field $k$ is {\it{ample}} (or {\it{large}}) if every smooth geometrically irreducible $k$-curve has either zero or infinitely many $k$-rational points. Ample fields, which were introduced by Pop in \cite{Pop96} (and which are necessarily infinite), include algebraically closed fields, some complete valued fields (e.g., $\Qq_p$, $\Rr$, $\kappa((T))$), the field $\Qq^{\rm{tr}}$ of all totally real algebraic numbers, etc. See \cite{Jar11}, \cite{BSF13}, and \cite{Pop14} for more details. Consequently, a special (but fundamental) case of Theorem \ref{thm:main} is that the Inverse Galois Problem has an affirmative answer over $\Hh(X)$, where $\Hh$ denotes the skew field of Hamilton's quaternions.

Given a skew field $H$ of finite dimension over its center $k$, with $k$ infinite, the ring $H[X]$ is one possible natural generalization of the usual polynomial ring in one variable over an infinite field. Another one is the polynomial ring $H_c[t]$, where $t$ is a central indeterminate, commuting with the coefficients\footnote{We adopt a different notation from that of \cite{DL19}, where this ring is denoted by $H[t]$, in order to distinguish between the cases of central variables and non-central ones. We note that there are alternative notations for this ring in the literature, such as $H[x, \id ,0]$ in \cite{Ore33}, or $H_L[t]$ in \cite{GM65}.} \footnote{Throughout this note, we use upper-case letters to denote non-central variables and lower case-letter to denote central ones, to add a visual distinction between the two.}. While these rings are isomorphic in the special case $H=k$, it is not clear that such an isomorphism exists if $H$ is non-commutative. This suggests that the Inverse Galois Problem over the field of fractions $H_c(t)$ of $H_c[t]$, which is studied by Deschamps and Legrand, and the same problem over $H(X)$ are a priori independent. In particular, although the Inverse Galois Problem over $H_c(t)$ has a positive answer if $k$ contains an ample field (see \cite[Th\'eor\`eme B]{DL19}), Theorem \ref{thm:main} has its own merits and, as \cite[Th\'eor\`eme B]{DL19}, extends the deep result of Pop solving the Inverse Galois Problem over the field $k(t)$, if $k$ contains an ample field.

We prove Theorem \ref{thm:main} in \S\ref{sec:proof}, by reducing it to the case settled by Deschamps and Legrand. The main observation needed is that the ring $H[X]$ is isomorphic to the ring $H_c[t_1, \dots, t_n]$ of polynomials over $H$ in $n$ central variables, where $n$ denotes the dimension of $H$ over its center (see Proposition \ref{prop:isom}). This follows from a theorem of Wilczynski \cite[Theorem 4.1]{Wil14}. We also make use of the general observation that the Inverse Galois Problem over skew fields is ``algebraic"; see Proposition \ref{prop:iso}.

\section{Polynomial rings and fields of fractions} \label{sec:basics}

\subsection{Polynomial rings} \label{ssec:basics1}

For this subsection, let $H$ be a skew field. 

The {\it{polynomial ring}} $H_c[t]$ in the {\it{central}} variable $t$ is the set of all sequences $(a_n)_{n \in \N}$ of elements of $H$ such that $a_n =0$ for all but finitely many $n$. As in the commutative setting, the addition is defined componentwise and the multiplication is defined by $(a_n)_n \cdot (b_n)_n = (c_n)_n$, where $c_n = \sum_{l+m=n} a_l b_m$ for every $n \in \N$. Setting $(a_n)_n = \sum_n a_n t^n$, one has $at=ta$ for every $a \in H$, thus justifying the terminology ``central". If $H$ is a field, then $H_c[t]$ is nothing but the usual polynomial ring in the variable $t$ over $H$. In the sense of Ore \cite{Ore33}, $H_c[t]$ is the skew polynomial ring $H[t, \alpha, \delta]$ in the variable $t$, where the automorphism $\alpha$ is the identity of $H$ and the derivation $\delta$ is 0. One can iteratively construct rings of polynomials in several central variables over $H$, by putting $H_c[t_1,t_2] = (H_c[t_1])_c[t_2]$, $H_c[t_1,t_2,t_3] = (H_c[t_1,t_2])_c[t_3]$, and so on. Since the variables are all central, the order in which they are added does not change the ring obtained, up to isomorphism. 

On the other hand, let $H\langle X \rangle$ be the free algebra in one symbol $X$ over $H$. That is, $H\langle X \rangle$ is the algebra spanned by all words whose letters are elements of $H$ or $X$. For an element $a \in H$ and $f(X) \in H\langle X \rangle$, the substitution $f(a) \in H$ is defined in the obvious way, by replacing each occurrence of $X$ in $f(X)$ by $a$, and computing the resulting value in $H$. For a fixed $a \in H$, the map $f(X) \mapsto f(a)$ is a homomorphism from $H\langle X \rangle$ to $H$. We say that $f$ {\it{vanishes}} at $a$ if $f(a) = 0$. Let $I$ be the (two-sided) ideal of $H\langle X \rangle$ which consists of all $f(X) \in H\langle X \rangle$ that vanish at all $a \in H$. Then the ring $H[X]$ is defined as the quotient $H\langle X \rangle/I$, and it is isomorphic to the ring of polynomial functions over $H$. Note that, if $H$ is an infinite field, then this definition coincides with the usual definition of the polynomial ring in the variable $X$ over $H$.

\subsection{Classical fields of fractions} \label{ssec:basics2}

For this subsection, let $R$ be a non-zero ring, not necessarily commu\-tative. Recall that $R$ is an {\it{integral domain}} if, for all $r \in R \setminus \{0\}$ and $s \in R \setminus \{0\}$, one has $sr \not=0$ and $rs \not=0$. From now on, suppose $R$ is an integral domain.

A {\it{classical right quotient ring for $R$}} is an overring $S \supseteq R$ such that every non-zero element of $R$ is invertible in $S$, and such that every element of $S$ can be written as $ab^{-1}$ for some $a \in R$ and some $b \in R \setminus \{0\}$. We say that $R$ is a {\it{right Ore domain}} if, for all non-zero elements $x$ and $y$ of $R$, there exist $r$ and $s$ in $R$ such that $xr = ys \not=0$. By \cite[Theorem 6.8]{GW04}, if $R$ is a right Ore domain, then $R$ has a classical right quotient ring $H$ which is a skew field and, by \cite[Proposition 1.3.4]{Coh95}, $H$ is unique up to isomorphism. We then say that $H$ is the {\it{classical right field of fractions}} of $R$.

If $H$ denotes an arbitrary skew field, then the polynomial ring $H_c[t]$ in the central variable $t$ over $H$ is an integral domain,  since the degree is additive on products. Moreover, $H_c[t]$ is a right Ore domain, by \cite[Theorem 2.6 and Corollary 6.7]{GW04}. By an easy induction, given a positive integer $n$, the polynomial ring $H_c[t_1, \dots, t_n]$ in $n$ central variables over $H$ has a classical right field of fractions, which we denote by $H_c(t_1, \dots, t_n)$.

\begin{proposition} \label{prop:same}
Let $H$ be a skew field and $n \geq 2$. Then the equality $H_c(t_1,\ldots,t_n) = (H_c(t_1, \dots, t_{n-1}))_c(t_n)$ holds.
\end{proposition}

\begin{proof} 
First, it is clear that the inclusion $H_c[t_1,\ldots,t_n] \subseteq (H_c(t_1, \dots, t_{n-1}))_c(t_n)$ holds. As every element of $H_c(t_1,\ldots,t_n)$ can be written as $fg^{-1}$ with $f$ and $g$ in $H_c[t_1,\ldots,t_n]$, we actually have $H_c(t_1,\ldots,t_n) \subseteq (H_c(t_1, \dots, t_{n-1}))_c(t_n)$. 

For the converse, take a polynomial $f = \sum_{l = 0}^m a_l t_n^l$ with $a_l \in H_c(t_1, \dots, t_{n-1})$ for every $l \in \{0, \dots, m\}$. As before, we can write $a_l = b_l c_l^{-1}$ with $b_l \in H_c[t_1, \dots, t_{n-1}]$ and $c_l \in H_c[t_1, \dots, t_{n-1}] \setminus \{0\}$, for every $l \in \{0, \dots, m\}$. Since $H_c[t_1, \dots, t_{n-1}] \subseteq H_c[t_1, \dots, t_n] \subseteq H_c(t_1, \dots, t_n)$ and $t_n \in H_c(t_1, \dots, t_n)$, we get that $f= \sum_{l = 0}^m b_l c_l^{-1} t_n^l$ is in $H_c(t_1, \dots, t_n)$. This shows the desired inclusion $(H_c(t_1, \dots, t_{n-1}))_c(t_n) \subseteq H_c(t_1,\ldots,t_n)$, since every element of $(H_c(t_1, \dots, t_{n-1}))_c(t_n)$ can be written as $f g^{-1}$ with $f$ and $g$ in $(H_c(t_1, \dots, t_{n-1}))_c[t_n]$.
\end{proof}

\begin{proposition} \label{prop:center}
Let $H$ be a skew field of center $k$ and let $n$ be a positive integer. The center of $H_c(t_1, \dots, t_n)$ equals $k(t_1, \dots, t_n)$. Moreover, if the dimension of $H$ over $k$ is finite, then the equality ${\rm{dim}}_{k(t_1, \dots, t_n)} H_c(t_1, \dots, t_n) = {\rm{dim}}_k H$ holds.
\end{proposition}

\begin{proof}
By, e.g., \cite[Proposition 2.1.5]{Coh95}, if $K$ is an arbitrary skew field of center $C$, then $C(t)$ is the center of $K_c(t)$. Hence, by iterating Proposition \ref{prop:same}, the center of $H_c(t_1, \dots, t_n)$ equals $k(t_1, \dots, t_n)$. Now, suppose ${\rm{dim}}_k H$ is finite. Then, by \cite[Proposition 9]{DL19}, we have $H_c(t_1) \cong H \otimes_k k(t_1)$. Consequently, ${\rm{dim}}_{k(t_1)} H_c(t_1)$ is finite and equals ${\rm{dim}}_k H$. As before, it remains to iterate Proposition \ref{prop:same} to conclude the proof.
\end{proof}

\begin{proposition}\label{prop:isom} 
Let $H$ be a skew field of finite dimension $n$ over its center $k$. Assume $k$ is infinite. Then the ring $H[X]$ is isomorphic to $H_c[t_1,\dots, t_n]$. 
\end{proposition}

\begin{proof} 
The existence of such an isomorphism follows from \cite[Theorem 4.1]{Wil14}. See also \cite[Theorem 5]{AP19} for a different, more explicit proof. For the convenience of the reader, we include an elementary proof in the special case $H= \Hh$, where $\Hh$ is the skew field of Hamilton's quaternions.

One has the following classical identity for each $a \in \Hh$: 
$${\rm{Re}}(a) = {1 \over 4}(a-iai-jaj-kak),$$
where ${\rm{Re}}(a)$ is the real component of $a$. More generally, putting 
$$y_1 = {1 \over 4}(X-iXi-jXj-kXk),$$ 
$$y_2 = {1 \over 4}(jXk-Xi-iX-kXj),$$ 
$$y_3 = {1 \over 4}(kXi-Xj-jX-iXk),$$ 
$$y_4 = {1 \over 4}(iXj-Xk-kX-jXi),$$ 
the functions $y_1,y_2,y_3,y_4 \in \Hh[X]$ obtain real values only, and one has $X = y_1+iy_2+jy_3+ky_4$. In particular, $y_1,y_2,y_3,y_4$ belong to the center of $\Hh[X]$, and we may then define a homomorphism $\phi \colon \Hh_c[t_1,t_2,t_3,t_4] \to \Hh[X]$ by $\phi(t_l) = y_l$, $1 \leq l \leq 4$, and $\phi(a) = a$ for all $a \in \Hh$. The equality $X = y_1+iy_2+jy_3+ky_4$ implies that $\phi$ is surjective.

Let $p=p(t_1, t_2, t_3, t_4) \in \Hh_c[t_1,t_2,t_3,t_4]$. By decomposing the coefficients of $p$ into their real, $i$, $j$, and $k$ components, we may present $p$ in the form $p = p_1 + p_2 i + p_3 j +p_4 k$ with $p_1,p_2,p_3,p_4 \in \R[t_1,t_2,t_3,t_4]$. If $p \neq 0$, then $p_l \neq 0$ for some $1 \leq l \leq 4$. Then there exists a non-zero tuple $a = (a_1,a_2,a_3,a_4) \in \R^4$ such that $p_l(a) \neq 0$. Hence, $\phi(p)$ does not vanish at $X = a_1+a_2i+a_3j+a_4k$, thus showing that $\phi$ is also injective. 
\end{proof}

\begin{corollary} \label{coro:isom}
Let $H$ be a skew field of finite dimension $n$ over its center $k$. Assume $k$ is infinite. Then the ring $H[X]$ has a classical right field of fractions, denoted by $H(X)$, which is isomorphic to $H_c(t_1, \dots, t_n)$.
\end{corollary}

\begin{proof}
As recalled, the ring $H_c[t_1, \dots, t_n]$ is a right Ore domain. Since $H[X]$ is isomorphic to $H_c[t_1, \dots, t_n]$ by Proposition \ref{prop:isom}, $H[X]$ is also a right Ore domain and so has a classical right field of fractions. Finally, \cite[\S1.3]{Coh95} shows that the isomorphism $H[X] \cong H_c[t_1, \dots, t_n]$ from Proposition \ref{prop:isom} extends to an isomorphism $H(X) \cong H_c(t_1, \dots, t_n)$.
\end{proof}

\section{Proof of Theorem \ref{thm:main}} \label{sec:proof}

We first make the general observation that the Inverse Galois Problem over skew fields is an ``algebraic problem". More precisely:

\begin{proposition} \label{prop:iso}
Let $H_1$ and $H_2$ be isomorphic skew fields and let $G$ be a finite group. Then there exists a Galois extension of $H_1$ of group $G$ if and only if there exists a Galois extension of $H_2$ of group $G$.
\end{proposition}

\begin{proof}
Let $\varphi : H_1 \rightarrow H_2$ be an isomorphism. Suppose there exists a Galois extension $K_1/H_1$ of group $G$.

By the exchange principle\footnote{which asserts that, given two sets $A$ and $B$, there exists a set $C$ such that $A \cap C= \emptyset$ and $|C|=|B|$ (see, e.g., \cite[page 31]{Vau95}).}, there exists a set $C$ such that $C\cap H_{2}=\emptyset$
and $|C|=|K_{1} \setminus H_{1}|$. Let $f : K_{1}\setminus H_{1} \rightarrow C$ be a bijection. Then set $K_2 = C \cup H_2$ and consider the well-defined map $\psi : K_1 \rightarrow K_2$ given by $\psi(x)= \varphi(x)$ if $x \in H_1$ and $\psi(x) = f(x)$ if $x \in K_1 \setminus H_1$. The map $\psi$ is surjective and, as $C\cap H_{2}=\emptyset$, it is also injective. Now, define the ring operations on $K_{2}$ as inherited from $K_{1}$ via $\psi$:
\[\forall x,y\in K_{2}, \, \, x \cdot y=\psi(\psi^{-1}(x) \cdot \psi^{-1}(y)),\,\,x+y=\psi(\psi^{-1}(x)+\psi^{-1}(y)).\]
Then $K_2$ is isomorphic to $K_1$ via $\psi$ and, in particular, $K_2$ is a skew field containing $H_2$.

It remains to show that $K_2/H_2$ is Galois of group $G$. To that end, note that the isomorphism $\psi : K_1 \rightarrow K_2$, whose restriction to $H_1$ equals $\varphi$, induces an isomorphism $\phi : {\rm{Aut}}(K_1/H_1) \rightarrow {\rm{Aut}}(K_2/H_2)$ (namely, $\phi(\sigma) = \psi \circ \sigma \circ \psi^{-1}$ for every $\sigma \in {\rm{Aut}}(K_1/H_1)$). Finally, if $x$ is any element of $K_2$ such that $\sigma(x) = x$ for every $\sigma \in {\rm{Aut}}(K_2/H_2)$, then we have $\tau(\psi^{-1}(x)) = \psi^{-1}(x)$ for every $\tau \in {\rm{Aut}}(K_1/H_1)$. As $K_1/H_1$ is Galois, we then have $\psi^{-1}(x) \in H_1$, and so $x \in H_2$, thus showing that $K_2/H_2$ is Galois. This concludes the proof.
\end{proof}

\begin{proof}[Proof of Theorem \ref{thm:main}]
By Corollary \ref{coro:isom}, we have $H(X) \cong H_c(t_1, \dots, t_n)$, where $n$ denotes the dimension of $H$ over $k$. Moreover, by Proposition \ref{prop:center}, the center of $H_c(t_1, \dots, t_{n-1})$ equals $k(t_1, \dots, t_{n-1})$ and the dimension of $H_c(t_1, \dots, t_{n-1})$ over $k(t_1, \dots, t_{n-1})$ is finite. Finally, $k(t_1, \dots, t_{n-1})$ contains an ample field. Hence, by \cite[Th\'eor\`eme B]{DL19}, the Inverse Galois Problem has an affirmative answer over the skew field $(H_c(t_1, \dots, t_{n-1}))_c(t_n)$, that is, over $H_c(t_1, \dots, t_{n})$ by Proposition \ref{prop:same}. It then remains to apply Proposition \ref{prop:iso} to get that the Inverse Galois Problem also has an affirmative answer over $H(X)$, thus concluding the proof.
\end{proof}

\begin{remark}
Similarly, we have this result, which follows from \cite[Proposition 12]{DL19} as Theorem \ref{thm:main} follows from \cite[Th\'eor\`eme B]{DL19}:

\vspace{2mm}

\noindent
{\it{Let $G$ be a finite group and $H$ a skew field of finite dimension $n$ over its center $k$. In each of the following cases, $G$ occurs as the Galois group of a Galois extension of $H(X)$:

\vspace{0.5mm}

\noindent
{\rm{(1)}} $G$ is abelian and $k$ is infinite,

\vspace{0.5mm}

\noindent
{\rm{(2)}} $G=S_m$ ($m \geq 3$) and $k$ is infinite,

\vspace{0.5mm}

\noindent
{\rm{(3)}} $G=A_m$ ($m \geq 4$) and $k$ has characteristic zero,

\vspace{0.5mm}

\noindent
{\rm{(4)}} $G$ is solvable, $n \geq 2$, and $k$ has positive characteristic.}}
\end{remark}

\bibliography{Biblio2}

\begin{thebibliography}{BSF13}

\bibitem[AP19]{AP19}
Gil Alon and Elad Paran.
\newblock A quaternionic {N}ullstellensatz.
\newblock {\em Manuscript}, 2019.
\newblock arXiv:1911.10595.

\bibitem[BSF13]{BSF13}
Lior Bary-Soroker and Arno Fehm.
\newblock Open problems in the theory of ample fields.
\newblock In {\em Geometric and differential {G}alois theories}, volume~27 of
  {\em S\'emin. Congr.}, pages 1--11. Soc. Math. France, Paris, 2013.

\bibitem[Coh95]{Coh95}
Paul~Moritz Cohn.
\newblock {\em Skew fields. {T}heory of general division rings}.
\newblock Encyclopedia of {M}athematics and its {A}pplications, 57. Cambridge
  {U}niversity {P}ress, Cambridge, 1995.
\newblock xvi + 500 pp.

\bibitem[DL19]{DL19}
Bruno Deschamps and Fran\c{c}ois Legrand.
\newblock Le probl\`eme inverse de {G}alois sur les corps des fractions tordus
  \`a ind\'etermin\'ee centrale. ({F}rench).
\newblock 2019.
\newblock To appear in Journal of Pure and Applied Algebra.
  \url{https://doi.org/10.1016/j.jpaa.2019.106240}.

\bibitem[GM65]{GM65}
Basil Gordon and Theodore~S. Motzkin.
\newblock On the zeros of polynomials over division rings.
\newblock {\em Trans. Amer. Math. Soc.}, 116:218--226, 1965.

\bibitem[GW04]{GW04}
Kenneth~R. Goodearl and Jr. Warfield, Robert~Breckenridge.
\newblock {\em An Introduction to noncommutative Noetherian rings}.
\newblock London Mathematical Society Student Texts, 61. Cambridge University
  Press, Cambridge, 2004.
\newblock Second edition. xxiv+344 pp.

\bibitem[Jar11]{Jar11}
Moshe Jarden.
\newblock {\em Algebraic patching}.
\newblock Springer {M}onographs in {M}athematics. Springer, Heidelberg, 2011.
\newblock xxiv + 290 pp.

\bibitem[Ore33]{Ore33}
Oystein Ore.
\newblock Theory of non-commutative polynomials.
\newblock {\em Ann. of Math. (2)}, 34(3):480--508, 1933.

\bibitem[Pop96]{Pop96}
Florian Pop.
\newblock Embedding problems over large fields.
\newblock {\em Ann. of Math. (2)}, 144(1):1--34, 1996.

\bibitem[Pop14]{Pop14}
Florian Pop.
\newblock Little survey on large fields - old $\&$ new.
\newblock In {\em Valuation theory in interaction}, EMS Ser. Congr. Rep., pages
  432--463. Eur. Math. Soc., Z\"urich, 2014.

\bibitem[Vau95]{Vau95}
Robert~L. Vaught.
\newblock {\em Set theory. An introduction}.
\newblock Birkh\"{a}user Boston, Inc., Boston, MA, 1995.
\newblock Second edition. x+167 pp.

\bibitem[Wil14]{Wil14}
Dariusz~M. Wilczynski.
\newblock On the fundamental theorem of algebra for polynomial equations over
  real composition algebras.
\newblock {\em J. Pure Appl. Algebra}, 218:1195--1205, 2014.

\end{thebibliography}
\bibliographystyle{alpha}

\end{document}